\RequirePackage[l2tabu, orthodox]{nag}
\documentclass[11pt]{article}%

\usepackage{amssymb}
\usepackage{amsfonts}
\usepackage{amsmath}
\usepackage{dsfont}

\usepackage{amsthm}
\usepackage{hyperref}
\usepackage{cleveref}

\usepackage{float}
\usepackage{graphicx}
\usepackage{caption}

\usepackage[all]{xy}
\usepackage{tikz}

\usepackage{url}

\usepackage{color}
\usepackage{enumerate}
\usepackage[titletoc, title]{appendix}

\usepackage[latin1]{inputenc}
\usepackage[a4paper]{geometry}
\usepackage{microtype}

\newtheorem{theo}{Theorem}[section]

\newtheorem{lem}[theo]{Lemma}
\newtheorem{cor}[theo]{Corollary}

\numberwithin{equation}{section}

\newcommand{\EE}{\mathbb{E}}

\newcommand{\NN}{\mathbb{N}}

\newcommand{\PP}{\mathbb{P}}

\newcommand{\RR}{\mathbb{R}}

\newcommand{\Var}{\mathrm{Var}}
\newcommand{\red}[1]{{\color{red}#1}}
\newcommand{\hide}[1]{}

\title{Pre-freezing transition in Boltzmann-Gibbs measures \\ associated with log-correlated fields}
\author{Mo Dick Wong \\
Statistical Laboratory, University of Cambridge
}

\date{\today}
\begin{document}

\maketitle

\abstract{We consider Boltzmann-Gibbs measures associated with log-correlated Gaussian fields as potentials and study their multifractal properties which exhibit phase transitions. In particular, the pre-freezing and freezing phenomena of the annealed exponent, predicted by Fyodorov using a modified replica-symmetry-breaking ansatz, are generalised to arbitrary dimension and verified using results from Gaussian multiplicative chaos theory.}
%\setcounter{tocdepth}{2}
%\tableofcontents

\section{Introduction}

Let $\{X_\epsilon(\cdot)\}_{\epsilon \in (0,1]}$ be a collection of centred Gaussian fields defined on a domain $D \subset \RR^d$ containing $[-2,2]^d$,
with covariance kernels given by
\begin{equation}\label{eq:LGF}
K_\epsilon(x, y) := \EE[X_\epsilon(x) X_\epsilon(y)] = -\log \left(|x-y| + \epsilon\right) + g_\epsilon(x, y), \qquad x, y \in D
\end{equation}

\noindent where $g_\epsilon$ are continuous functions (in variables $x$ and $y$) that are uniformly bounded, in the sense that there exists a constant $C_g > 0$ independent of $\epsilon$ such that
\begin{align*}
\sup_{x, y \in D} |g_\epsilon(x,y)| \le C_g.
\end{align*}

\noindent Such collection of random fields, for instance, arises from the mollification of log-correlated fields: if $X$ is a centred Gaussian field such that $\EE[X(x) X(y)] = -\log |x-y| + g(x, y)$ for some bounded continuous function $g$, and $\varphi_\epsilon := \epsilon^{-d} \varphi(\cdot/\epsilon)$ where $\varphi$ is a mollifier, then $X \ast \varphi_\epsilon$ is a centred Gaussian field with covariance structure of the form \eqref{eq:LGF}.

For each $\epsilon > 0$, the continuity of $K_\epsilon$ suggests that $X_\epsilon$ may be defined pointwise, and we may study the behaviour of its exponentiation $e^{\beta X_\epsilon(\cdot)}$ as $\epsilon$ tends to $0$. Given that $\Var(X_\epsilon(\cdot))$ is blowing up everywhere, this problem is not well-posed unless we introduce some suitable normalisation. One possibility is to normalise $e^{\beta X_\epsilon(\cdot)}$ by its expectation, i.e. to consider the sequence of random measures
\begin{align} \label{eq:GMC}
M_{\beta, \epsilon}(dx) = e^{\beta X_\epsilon(x) - \frac{\beta^2}{2} \Var(X_\epsilon(x))}dx.
\end{align}

\noindent The limit object as $\epsilon$ goes to $0$ is called the Gaussian multiplicative chaos, which was first constructed in the 1980's by Kahane \cite{kahane} who would like to provide a mathematically rigorous framework for Kolmogorov's turbulence model for energy dissipation \cite{M1972}. The theory of multiplicative chaos has seen a lot of development in the past few years due to the interest from mathematical physics. Indeed, an important example of log-correlated Gaussian field is the Gaussian free field for $d=2$, and the associated chaos measure, known as the Liouville quantum gravity measure, plays a central role in random geometry alongside other probabilistic objects such as Schramm-Loewner evolution \cite{DS, BNotes}. More recently the theory has also shown up in other areas, for instance in random matrix theory where the characteristic polynomials of a large class of random matrices behave asymptotically like a multiplicative chaos \cite{CUE, Hermitian, LOS}, and in probabilistic number theory where the statistical behaviour of the Riemann zeta function on the critical line is found to be closely related to (complex) multiplicative chaos \cite{SW}. We refer the readers to \cite{RV2014} for a detailed review of the theory.

In this paper we shall adopt a different normalisation, and consider the Boltzmann-Gibbs distribution with potential $X_\epsilon$ on $[-1, 1]^d$, i.e.
\begin{align}\label{eq:Gibbs}
\mu_{\beta, \epsilon}(dx)
:= m_{\beta, \epsilon}(x)dx := \frac{1}{Z_{\epsilon}(\beta)}e^{\beta X_\epsilon(x)}dx,
\qquad Z_\epsilon(\beta):= \int_{[-1,1]^d} m_{\beta, \epsilon}(x)dx.
\end{align}

\noindent This is closely related to the previous normalisation as one can imagine that the properties of $\mu_{\beta, \epsilon}(dx)$ are similar, if not identical, to that of $\frac{M_{\beta, \epsilon}(dx)}{M_{\beta, \epsilon}([-1, 1]^d)}$. Unlike $M_{\beta, \epsilon}$ the limit of which is almost surely trivial when $\beta^2 \ge 2d$ from the standard theory of multiplicative chaos, $\{\mu_{\beta, \epsilon}\}_\epsilon$ are probability measures defined on a single compact set, and hence a non-trivial (distributional) limit of $\mu_{\beta, \epsilon}$ as $\epsilon \to 0$ (possibly along a subsequence) exists by Prokhorov's theorem.  

The study of \eqref{eq:Gibbs} is motivated by the interest from the physics community in understanding disorder-induced multifractality. The special case where $d = 2$ is of particular importance not only because of its connection to Liouville quantum gravity \cite{KMT1996} as we described earlier, but also other physical problems such as Dirac particles in a random magnetic field \cite{CMW1996}, see \cite{paper} for more examples and physical discussions.

\subsection{Main result}
The focus of this paper is the multifractal nature of the limiting Boltzmann-Gibbs measure, which is captured, for $q>1$, by the moments
\begin{align*}
\int_{[-1,1]^d} m_{\beta, \epsilon}(x)^q dx = \frac{Z_\epsilon(q\beta)}{Z_\epsilon(\beta)^q} &\underset{\epsilon \to 0}{\sim} \epsilon^{\eta_q},\\
\EE\left[\int_{[-1,1]^d} m_{\beta, \epsilon}(x)^q dx\right] = \EE\left[\frac{Z_\epsilon(q\beta)}{Z_\epsilon(\beta)^q}\right] &\underset{\epsilon \to 0}{\sim} \epsilon^{\tilde{\eta}_q}.
\end{align*}

\noindent We call $\eta_q, \tilde{\eta}_q$ the quenched and annealed multifractal exponents respectively. It is known that the quenched multifractal exponent $\eta_q$ exhibits the following phase transition:
\begin{align} \label{eq:eta_q}
\eta_q = 
\begin{cases} 
-\frac{\beta^2 q^2}{2} + \frac{\beta^2 q}{2}, & 0 \le \beta^2 \le \frac{2d}{q^2}, \\
d - \sqrt{2d} |\beta| q + \frac{\beta^2 q}{2}, & \frac{2d}{q^2} \le \beta^2 \le 2d, \\
-d(q-1), & \beta^2 \ge 2d. 
\end{cases}
\end{align}

\noindent The observation that $\eta_q$ attains a $\beta$-independent value in the low temperature regime $\beta^2 \ge 2d$ is known as \textbf{freezing} in the physics literature and is related to the fact that the measure becomes localised (see e.g. \cite{AZ2014, AZ2015} for more precise statements, and also our discussion in Section \ref{sec:ansatz}). We note that \eqref{eq:eta_q} is hardly surprising, at least at the heuristic level, given the properties of Gaussian multiplicative chaos and log-correlated Gaussian fields: 
\begin{itemize}
\item When $\beta^2 < \frac{2d}{q^2}$, we may rewrite the $q$-th moment, using the fact that $e^{\gamma \Var(X_\epsilon(x))} \asymp \epsilon^{-\gamma}$ for any $\gamma \in \RR$ and $x \in [-1, 1]^d$, as
\begin{align*}
\frac{Z_\epsilon(q\beta)}{Z_\epsilon(\beta)^q} \asymp \epsilon^{-\frac{\beta^2q^2}{2} + \frac{\beta^2 q}{2}} \frac{M_{\beta q, \epsilon} ([-1, 1]^d)}{M_{\beta, \epsilon}([-1, 1]^d)^q}.
\end{align*}

Suppose the fields $\{X_\epsilon(\cdot)\}_\epsilon$ are obtained from mollification. It is known that whenever $\gamma^2 < 2d$, i.e. within the subcritical regime of Gaussian multiplicative chaos, the sequence $M_{\gamma, \epsilon}(dx)$ defined in \eqref{eq:GMC} converges as $\epsilon \to 0$ to a non-trivial random measure $M_{\gamma}(dx)$ which is supported on the so-called $\gamma$-thick points of the field \cite{B2017}, i.e. points $x$ such that
\begin{align*}
\lim_{\epsilon \to 0} \frac{X_\epsilon(x)}{- \log \epsilon} = \gamma,
\end{align*} 

\noindent and the limiting measure satisfies $\EE[M_{\gamma}(A)] = |A|$ where $|A|$ is the Lebesgue measure of a Borel set $A$. Therefore we should expect that $M_{\beta q, \epsilon}([-1, 1]^d)$ and $M_{\beta, \epsilon}([-1,1]^d)$ to remain of order one as $\epsilon$ approaches zero, and the order of the $q$-th moment is precisely given by the simple scaling $\epsilon^{-\frac{\beta^2q^2}{2} + \frac{\beta^2 q}{2}}$.

\item When $\frac{2d}{q^2} \le \beta^2 < 2d$, the above argument does not apply anymore because $M_{\beta q, \epsilon}$ converges to a trivial limit as $\epsilon \to 0$. Therefore the correct way to view the $q$-th moment would be
\begin{align*}
\frac{Z_\epsilon(q\beta)}{Z_\epsilon(\beta)^q} \asymp \epsilon^{ \frac{\beta^2 q}{2}} \frac{\int_{[-1, 1]^d} e^{\beta q X_\epsilon(x)}dx}{M_{\beta, \epsilon}([-1, 1]^d)^q}.
\end{align*}

\noindent It is reasonable to expect that the size of the numerator is similar, at least in exponential scale, to that of the discrete sum $\epsilon^d \sum_{i} e^{\beta q X_\epsilon(r_i)}$ with $\{r_i\}$ being $\epsilon^{-d}$ equally-spaced points in $[-1,1]^d$. Note that for any $c >0$, the probability that $X_\epsilon(x)$ is at least $-c \log \epsilon$ is of order $\epsilon^{\frac{1}{2}c^2}$, and hence the discrete sum should be of order at least
\begin{align*}
\epsilon^{\frac{1}{2}c^2} e^{|\beta| q (-c \log \epsilon)} = \epsilon^{\frac{1}{2}(c - |\beta| q)^2 - \frac{\beta^2q^2}{2}}.
\end{align*}

\noindent The idea now is to make $(c - |\beta|q)^2$ small so that the bound is as tight as possible. It is, however, known that the ``maxima" of the field $X_\epsilon(\cdot)$ are of size $-\sqrt{2d} \log \epsilon$ as $\epsilon \to 0$ (see \cite{max} for more details), and hence the best one can achieve is $c = \sqrt{2d}$, which yields the desired order of the $q$-th moment.

\item When $\beta^2 \ge 2d$, one may understand the $q$-th moment directly by discretising the integrals in the numerator and the denominator. Our knowledge about the maxima of $X_\epsilon(\cdot)$, as explained above, will then lead to $\eta_q = -d(q-1)$.
\end{itemize}

\noindent Here we would like to study instead the annealed exponent $\tilde{\eta}_q$, which can behave rather differently from $\eta_q$ due to the fact that the expectation may be dominated by contributions from rare events. This is a reformulation of the problem considered in \cite{paper} where a precise prediction was stated for $\tilde{\eta}_q$ for $d=1$. Our goal is to validate the prediction and extend it to arbitrary dimension. As discussed below, there was in fact a minor inconsistency in the reported value of the exponent in the conjecture.

\begin{theo}\label{theo:main}
For $q > 1$,
\begin{align*}
\tilde{\eta}_q : = \lim_{\epsilon \to 0} \frac{\log \EE\left[ \frac{Z_\epsilon(q\beta)}{Z_\epsilon(\beta)^q}\right] }{\log \epsilon}
\end{align*}

\noindent exists and is given by
\begin{align} \label{eq:teta_q}
\tilde{\eta}_q = \begin{cases}
-\frac{\beta^2 q^2}{2} + \frac{\beta^2 q}{2}, & 0 \le \beta^2 < \frac{2d}{2q-1},\\
\frac{(2d - \beta^2)^2}{8\beta^2} - d(q-1), & \frac{2d}{2q-1} \le \beta^2 < 2d,\\
-d(q-1), & \beta^2 \ge 2d.
\end{cases}
\end{align}

\end{theo}

We note that our exponent in the intermediate regime $\frac{2d}{2q-1} \le \beta^2 < 2d$ is slightly different from that in the original prediction. When $d = 1$, our annealed exponent $\tilde{\eta}_q$ is related to $\tilde{\tau}_q$ in \cite{paper} by $\tilde{\eta}_q = \tilde{\tau}_q - (q-1)$ and that our $\beta^2$ is equal to $2 \gamma$ in the notation of that paper. The value of $\tilde{\tau}_q$ in the intermediate regime was incorrectly reported as $(1-\gamma)^2/2\gamma$ instead of $(1-\gamma)^2/4\gamma$ in the paper except equation (9). This seems to have been caused by a systematic misprint.

Theorem \ref{theo:main} suggests that $ \EE\left[ Z_\epsilon(q\beta)/Z_\epsilon(\beta)^q\right] = \epsilon^{\tilde{\eta}_q + o(1)}$, and it is natural to ask what subleading order terms are hidden in the $o(1)$ term. For this we would like to mention a recent paper \cite{CRSLD2017} where the authors, by considering a directed polymer model on Cayley trees, conjectured an intriguing second order phase transition in the subcritical regime $\beta^2 < 2d$, with the subleading order term depending on whether $\beta^2$ is smaller than, equal to or bigger than $\frac{2d}{2q-1}$ (see equation (20)). It is not clear if our method in the current paper can be adapted to the analysis of  second order behaviour, which will require more refined versions of Lemma \ref{lem:1} and Lemma \ref{lem:2} even if that is possible, and this new prediction will be the subject of a future investigation.

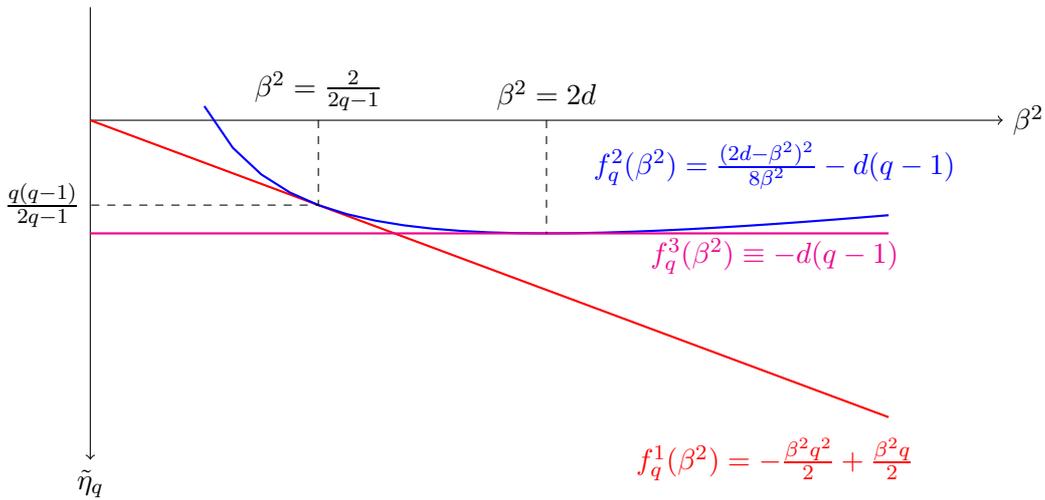
\begin{figure}[h!]
\begin{center}
\begin{tikzpicture}[scale=3]
\draw[->] (0,0) -- (4,0) node[right] {$\beta^2$};
\draw[->] (0,0.5) -- (0, -1.5) node[below] {$\tilde{\eta}_q$};
%implicitly assume q = 1.5
\draw[thick, red, domain=0:3.5] plot (\x, -0.375*\x); \node at (3, -1.5) {\color{red} $f^1_q(\beta^2) = -\frac{\beta^2 q^2}{2} + \frac{\beta^2 q}{2}$};
\draw[thick, blue, domain=0.5:3.5] plot (\x, {(2-\x)*(2-\x)/(8*\x) -0.5}); \node at (3, -0.2) {\color{blue} $f^2_q(\beta^2) = \frac{(2d-\beta^2)^2}{8\beta^2} - d(q-1)$};
\draw[thick, magenta, domain=0:3.5] plot(\x, -0.5); \node at (3, -0.6) {\color{magenta} $f^3_q(\beta^2) \equiv -d(q-1)$};
\draw[dashed] (1, 0) node[above] {$\beta^2 = \frac{2}{2q-1}$} -- (1, -0.375) -- (0, -0.375) node[left] {$\frac{q(q-1)}{2q-1}$};
\draw[dashed] (2, 0) node[above] {$\beta^2 = 2d$} -- (2, -0.5);
\end{tikzpicture}
\end{center}
\caption{\label{fig:phase} Behaviour of $\tilde{\eta}_q$ in three regimes.}
\end{figure}

\subsection{Relation of the main result to replica symmetry breaking ansatz} \label{sec:ansatz}

Comparing \eqref{eq:eta_q} and \eqref{eq:teta_q}, one can see that $\eta_q$ and $\tilde{\eta}_q$ behave quite differently in the subcritical regime $\beta^2 < 2d$ in two ways: the simple scaling of the annealed exponent $\tilde{\eta}_q = -\frac{\beta^2 q^2}{2} + \frac{\beta^2 q}{2}$ remains true up to  $\beta^2 = \frac{2d}{2q-1}$, beyond which an entirely different expression appears. For a physical interpretation of \eqref{eq:teta_q} (and \eqref{eq:eta_q} similarly), let us focus on the case where $q$ is a positive integer. Suppose we partition $[-1,1]^d$ into $\epsilon^{-d}$ boxes $\{B_i\}$ of width $\epsilon$, and let $x_i$ be the centre of the box $B_i$, then it is reasonable to expect that
\begin{align*}
\int_{[-1,1]^d} m_{\beta, \epsilon}^q(x) dx 
= \frac{Z_\epsilon(q\beta)}{Z_\epsilon(\beta)^q} 
= \frac{\int_{[-1,1]^d} e^{\beta q X_\epsilon(u)} du}{\left(\int_{[-1,1]^d} e^{\beta X_\epsilon(x)} dx\right)^q}
\approx \epsilon^{-d(q-1)} \sum_i \left(\frac{e^{\beta X_\epsilon(r_i)}}{\sum_j e^{\beta X_\epsilon(r_j)}}\right)^q.
\end{align*}

\noindent In other words the $q$-th moments are related to a discrete Boltzmann-Gibbs measure defined on the points $\{x_i\}$ with the same potential $X_\epsilon(\cdot)$. One may interpret $\sum_i \left(\frac{e^{\beta X_\epsilon(r_i)}}{\sum_j e^{\beta X_\epsilon(r_j)}}\right)^q$ as the probability of getting $q$ identical configurations when we draw $q$ independent samples from a discrete Boltzmann-Gibbs distribution, and Theorem \ref{theo:main} suggests that

\begin{align*}
\bar{p}_\epsilon(q) := \EE \left[\sum_i \left(\frac{e^{\beta X_\epsilon(r_i)}}{\sum_j e^{\beta X_\epsilon(r_j)}}\right)^q \right] \sim \epsilon^{\hat{\eta}_q}, 
\qquad \hat{\eta}_q
= \begin{cases}
-\frac{\beta^2 q^2}{2} + \frac{\beta^2 q}{2} + d(q-1), & 0 \le \beta^2 < \frac{2d}{2q-1},\\
\frac{(2d - \beta^2)^2}{8\beta^2}, & \frac{2d}{2q-1} \le \beta^2 < 2d,\\
0 , & \beta^2 \ge 2d.
\end{cases}
\end{align*}

\noindent In the low temperature regime $\beta^2 \ge 2d$, we see that the Boltzmann-Gibbs measure is so localised that the probability of having $q$ identical configurations is non-trivial\footnote{This does not really follow from our theorem, which only suggests that $\hat{\eta}_q = o(1)$ in the freezing regime, but is nevertheless shown to be true in e.g. \cite{AZ2014, AZ2015}.} even in the limit. While $\bar{p}_\epsilon(q)$ is tending to $0$ in the subcritical case $\beta^2 < 2d$, the rate at which it converges to $0$ suddenly does not depend on $q$ anymore if $q$ is sufficiently large (so long as $\beta^2 \ne 0$). Such a phenomenon is called \textbf{pre-freezing} by Fyodorov in \cite{paper}.

Physically, the result suggests that the Bolzmann-Gibbs measure cannot simply be studied by the standard \textbf{replica symmetry breaking (RSB)} methodology in the intermediate regime. This is already apparent in the original argumentation of Fyodorov. Let us summarise a few key ideas from that paper even though they will not be needed in the following. Employing the replica method, one would like to evaluate the RHS of
\begin{align} \label{eq:replica}
\EE\left[\frac{Z_{\epsilon}(q \beta)}{Z_\epsilon(\beta)^q}\right] = \lim_{n \to 0} \EE\left[Z_\epsilon(q \beta) Z_\epsilon(\beta)^{n-q}\right]
\end{align}

\noindent as if $n-q$ were a positive integer. Following \cite{FB2008}, one considers an ``infinite-dimensional limit''\footnote{This is rather remarkable given \eqref{eq:teta_q} because with this scaling, $\tilde{\eta}_q$ is exactly proportional to $d$ in all the regimes.} (by which the physicists mean letting $\beta^2 = \beta_0^2 d$ and sending $d \to \infty$) and rewrites \eqref{eq:replica} as an integral over a matrix variable $Q$ that is related to the overlap. In this expression, Fyodorov views $Z_\epsilon(q\beta)$ as one replica and $Z_\epsilon(\beta)^{n-q}$ as $n-q$ other replicas, and now the crucial observation in \eqref{eq:replica} is that the first replica can affect the behaviour of the subsequent ones. Indeed, the special replica $Z_\epsilon(q \beta)$ behaves as if its effective temperature was much colder, by a factor of $q$, allowing to ``pre-freeze" the subsequent replicas. To phrase it slightly differently, some of the $n-q$ high-temperature replicas may form a cluster with higher overlap among themselves as well as with the low-temperature replica. Using such a modified 1-RSB solution for the saddle point method indeed leads to the right prediction in the pre-freezing regime $\beta^2 \in \left[\frac{2d}{2q-1}, 2d\right)$.

Our results are also related to the papers \cite{AZ2014, AZ2015} by Arguin and Zindy where the authors study essentially the same Boltzmann-Gibbs measure in one and two dimensions. By adapting the Bovier-Kurkova technique from generalised random energy models, they are able to prove that the overlap of two points sampled from the Gibbs measure at low temperature is either $0$ or $1$. This is then used to show that the joint distribution of Gibbs weights has a Poisson--Dirichlet statistics in the limit, which in particular verifies a conjecture by Carpentier and Le Doussal \cite{CLD2001} regarding the statistics of the extrema of a log-correlated field.

\subsection{Intuition for the main result}

It is unclear how to make the replica calculations mathematically rigorous. We do not attempt to justify the method in \cite{paper}, and the proof of Theorem \ref{theo:main} shall follow an entirely different approach. We also have a different angle to interpret the result in the pre-freezing regime.  For every $c \in \RR$, it is possible to rewrite
\begin{align}\label{eq:interpret2}
\EE \left[ \frac{Z_\epsilon(q\beta)}{Z_\epsilon(\beta)^q} \right]
& \asymp \int_{[-1,1]^d} \epsilon^{-\frac{\beta^2 q^2 c^2}{2}+\frac{\beta^2 q}{2}}  \EE \left[e^{\beta q c X_\epsilon(u) - \frac{\beta^2 q^2 c^2}{2}\Var(X_\epsilon(u))} \frac{e^{\beta q(1-c) X_\epsilon(u)}}{M_{\beta, \epsilon}([-1, 1]^d)^q}\right] du.
\end{align}

\noindent Note that the factor $e^{\beta qc X_\epsilon(u) - \frac{\beta^2 q^2 c^2}{2} \Var(X_\epsilon(u))}$ may be seen as a change of measure (by Cameron-Martin theorem, see also Lemma \ref{lem:Girsanov}). In other words, to understand the average size of $Z_\epsilon(q \beta) / Z_\epsilon(\beta)^q$, we sample a point $u$ uniformly in $[-1,1]^d$ and study the size of
\begin{align} \label{eq:interpret3}
\epsilon^{-\frac{\beta^2 q^2 c^2}{2}+\frac{\beta^2 q}{2}}\frac{e^{\beta q(1-c) X_\epsilon(u)}}{M_{\beta, \epsilon}([-1, 1]^d)^q}
\end{align}

\noindent as we shift the mean of $X_\epsilon(u)$ to $\beta q c \Var(X_\epsilon(u)) \asymp -\beta q c \log \epsilon$. If we choose $c$ suitably so that $c = c_0 = \frac{1}{q} \left(\frac{1}{2} + \frac{d}{\beta^2}\right)$ and believe that $M_{\beta, \epsilon}([-1,1]^d)$ is still roughly of order one after the change of measure (this is essentially the content of Lemma \ref{lem:1}), we see by plugging in $X_\epsilon(u) \approx  -\beta q c \log \epsilon$ that the RHS of \eqref{eq:interpret2} should be of order
\begin{align*}
\epsilon^{-\frac{\beta^2 q^2 c^2}{2}} \epsilon^{\frac{\beta^2 q}{2}} e^{\beta q(1-c) \left(-\beta q c \log \epsilon\right)}
= \epsilon^{\frac{(2d-\beta^2)^2}{8\beta^2} - d(q-1)}
\end{align*}

\noindent which is exactly what we want. On the other hand if we apply the heuristic
\begin{align*}
M_{\beta, \epsilon}(B(u, \epsilon)) \underset{\epsilon \to 0}{\sim} \epsilon^{d + \frac{\beta^2}{2}} e^{\beta X_\epsilon(u)},
\end{align*}

\noindent then this particular choice of $c$ results in a change of measure such that $M_{\beta, \epsilon}(B(u, \epsilon))$ becomes roughly of order one, i.e. the Gaussian multiplicative chaos measure $M_{\beta, \epsilon}(\cdot)$ has mass localised at $u$. In other words localisation, despite being extremely rare (and indeed nonexistent in the limit as $\epsilon \to 0$), is responsible for the behaviour of $\tilde{\eta}_q$ in the subcritical regime $\beta^2 < 2d$ when $q$ is sufficiently large. This mechanism does not come into play when $q$ is so small (relative to $\beta^2$) that $c_0 > 1$, however, because simple scaling dominates the expectation and we are better off stopping at $c = 1$ instead when $\beta^2 < \frac{2d}{2q-1}$.

\paragraph{Organisation of paper} The remainder of the paper is organised as follows. In Section \ref{sec:2} we commence by explaining the main ideas of our proof, and in particular how we came up with the suitable choice of $c$ in the subcritical regime. This is followed by results regarding the integrability of Gaussian multiplicative chaos and supremum of Gaussian processes which are instrumental in our proof. The proof of Theorem \ref{theo:main} is then presented in Section \ref{sec:proof}. We also include some further Gaussian comparison results in the Appendix.

\paragraph{Acknowledgement} I would like to thank Nathana\"el Berestycki and Yan Fyodorov for suggesting this problem and for useful comments on preliminary drafts of this article. I am supported by the Croucher Foundation Scholarship and the EPSRC grant EP/L016516/1 for my PhD study at the Cambridge Centre for Analysis.

\section{Preliminaries} \label{sec:2}

\subsection{Change of measure and main ideas}
As mentioned in the introduction, we exploit the idea of change of measure via Cameron-Martin theorem. Due to its importance let us first record it carefully in the following Lemma.

\begin{lem}\label{lem:Girsanov}
Let $c \in \RR$. Then 
\begin{align}\label{eq:Girsanov}
\EE\left[ \frac{Z_\epsilon(\beta q)}{Z_\epsilon(\beta)^q} \right] 
\asymp \epsilon^{-\frac{\beta^2 q^2}{2} [1-(1-c)^2] + \frac{\beta^2 q}{2}} \int_{[-1,1]^d}\EE \left[\frac{e^{\beta q(1-c) X_{\epsilon}(u)}}{\left(\int_{[-1,1]^d} \frac{e^{\beta X_\epsilon(x) - \frac{\beta^2}{2} \Var \left(X_\epsilon(x)\right)}}{(|x-u| + \epsilon)^{\beta^2 q c}}dx \right)^q}\right]du.
\end{align}
\end{lem}

\begin{proof}
Recalling \eqref{eq:LGF} that we have $e^{\gamma \Var(X_\epsilon(x))} \asymp \epsilon^{-\gamma}$ for any $\gamma \in \RR$ and $x \in [-1, 1]^d$. Then for any $c \in \RR$,
\begin{align*}
\EE\left[ \frac{Z_\epsilon(\beta q)}{Z_\epsilon(\beta)^q} \right] 
& = \EE \left[ \int_{[-1,1]^d} \frac{e^{\beta q X_\epsilon(u)}}{\left(\int_{[-1,1]^d} e^{\beta X_\epsilon(x)}dx\right)^q}du\right]\\
& \asymp \epsilon^{-\frac{\beta^2 q^2 c^2}{2}} \int_{[-1,1]^d} \EE \left[e^{\beta q c X_\epsilon(u) - \frac{\beta^2 q^2 c^2}{2} \Var(X_\epsilon(u))} \frac{e^{\beta q(1-c)X_\epsilon(u)}}{\left(\int_{[-1,1]^d} e^{\beta X_\epsilon(x)}dx\right)^q}\right] du\\
& = \epsilon^{-\frac{\beta^2 q^2 c^2}{2}} \int_{[-1,1]^d} \EE \left[\frac{e^{\beta q(1-c)X_\epsilon(u) + \beta^2 q^2 c(1-c) \EE[X_\epsilon(u)^2]}}{\left(\int_{[-1,1]^d} e^{\beta X_\epsilon(x) + \beta^2 q c \EE[X_\epsilon(x) X_\epsilon(u)]}dx\right)^q}\right] du\\
& \asymp \epsilon^{-\frac{\beta^2 q^2 c^2}{2} - \beta^2 q^2 c(1-c)} \int_{[-1,1]^d} \EE \left[\frac{e^{\beta q(1-c)X_\epsilon(u)}}{\left(\int_{[-1,1]^d} \frac{e^{\beta X_\epsilon(x)}}{(|x-u|+\epsilon)^{\beta^2 q c}} dx\right)^q}\right] du\\
& \asymp
\epsilon^{-\frac{\beta^2 q^2}{2} [1-(1-c)^2] + \frac{\beta^2 q}{2}} \int_{[-1,1]^d}\EE \left[\frac{e^{\beta q(1-c) X_{\epsilon}(u)}}{\left(\int_{[-1,1]^d} \frac{e^{\beta X_\epsilon(x) - \frac{\beta^2}{2} \Var \left(X_\epsilon(x)\right)}}{(|x-u| + \epsilon)^{\beta^2 q c}}dx \right)^q}\right]du.
\end{align*}

\end{proof}

The proof of Theorem \ref{theo:main} will be divided into three parts, each corresponding to a different regime. In each regime the goal is to prove matching lower and upper bounds for $\tilde{\eta}_q$ separately. With Lemma \ref{lem:Girsanov}, our task is to choose the parameter $c \in \RR$ carefully so that if we have a good uniform control over the integrand, we should be able to obtain the desired bounds. 

\begin{itemize}
\item In the high temperature regime where $\beta^2 < \frac{2d}{2q-1}$, it is predicted that $\tilde{\eta}_q = -\frac{\beta^2q^2}{2} + \frac{\beta^2 q}{2}$. Even though this simple scaling is true for $\eta_q$ only in a smaller interval $\beta^2 < \frac{2d}{q^2}$, one may guess that the mechanism that explains this behaviour is the same in both cases. Given the thick point picture explained in \cite{B2017}, it is reasonable to choose $c = 1$.

\item In the intermediate regime where $\frac{2d}{2q-1} \le \beta^2 < 2d$,  there is a competition between the simple scaling term $\epsilon^{-\frac{\beta^2 q^2}{2} [1-(1-c)^2] + \frac{\beta^2 q}{2}}$ and the integral. Suppose we were allowed to ignore the numerator $e^{\beta q(1-c) X_\epsilon(u)}$ in the expectation in \eqref{eq:Girsanov} for $c \in [0, 1]$. After all if it were too large we would probably expect the denominator to be large as well, and such event has low probability anyway. Then the size of $\EE \left[Z_\epsilon(q\beta) / Z_\epsilon(\beta)^q\right]$ might be similar to that of
\begin{align}\label{eq:heuristic}
& \epsilon^{-\frac{\beta^2 q^2}{2} [1-(1-c)^2] + \frac{\beta^2 q}{2}} \int_{[-1,1]^d}\EE \left[\left(\int_{[-1,1]^d} \frac{e^{\beta X_\epsilon(x) - \frac{\beta^2}{2} \Var \left(X_\epsilon(x)\right)}}{(|x-u| + \epsilon)^{\beta^2 q c}}dx \right)^{-q}\right]du.
\end{align}

The negative $q$-th moment in \eqref{eq:heuristic} is unlikely to blow up as $\epsilon \to 0$, but we may be worried that it will vanish, or equivalently that
\begin{align*}
\int_{[-1,1]^d} \frac{e^{\beta X_\epsilon(x) - \frac{\beta^2}{2} \Var \left(X_\epsilon(x)\right)}}{(|x-u| + \epsilon)^{\beta^2 q c}}dx
\end{align*}

\noindent is exploding in the limit. To rule out such possibility, it suffices to show that 

\begin{align}\label{eq:integrand2}
\EE \left[\left(\int_{[-1,1]^d} \frac{e^{\beta X_\epsilon(x) - \frac{\beta^2}{2} \Var \left(X_\epsilon(x)\right)}}{(|x-u| + \epsilon)^{\beta^2 q c}}dx \right)^{t}\right]
\end{align}

\noindent is upper-bounded by a constant independent of $\epsilon$ for some $t > 0$. From Lemma \ref{lem:1}(i) which we shall present in the next subsection, this is actually true whenever 
\begin{align}\label{eq:heuristic2}
\beta^2 q c - d + \frac{\beta^2}{2}(t-1) < 0. 
\end{align}

We may ignore $t$ since it can be taken to be arbitrarily small. Cheating further, we replace \eqref{eq:heuristic2} by its nonstrict version and consider the following optimisation problem:
\begin{align*}
\text{minimise} \quad  f(c) := -\frac{\beta^2 q^2}{2}[1 - (1-c)^2] + \frac{\beta^2 q}{2}
\qquad \text{s.t.} \quad \beta^2 q c - d - \frac{\beta^2}{2} \le 0, \quad c \in [0, 1].
\end{align*}

Solving this optimisation yields $c = c_0 =  \frac{1}{q}\left(\frac{1}{2} + \frac{d}{\beta^2}\right)$ which happens to be the perfect choice because $f(c_0) = \frac{(2d - \beta^2)^2}{8\beta^2} - d(q-1)$ as explained earlier. When we prove the upper bound for $\tilde{\eta}_q$, the idea of ignoring $e^{\beta q(1-c) X_\epsilon(u)}$ and studying \eqref{eq:integrand2} instead of the negative $q$-th moment will be pursued by means of reverse H\"older's and Jensen's inequalities.

\item In the low temperature regime where $\beta^2 \ge 2d$, there is no way to upper bound \eqref{eq:integrand2} by a constant independent of $\epsilon$ anymore, but this can be replaced by an estimate in Lemma \ref{lem:1}(ii) and we just have to find some $c \in \RR$ that can give us a matching lower bound for $\EE[Z_\epsilon(q\beta) / Z_\epsilon(\beta)^q]$ (hence a matching upper bound for $\tilde{\eta}_q$). As for the complementary bound, it happens that the answer can be obtained directly using the physical observation that the Boltzmann-Gibbs measure is localised.
\end{itemize}

We shall now discuss the results needed to realize the aforementioned strategy.
 
\subsection{Integrability of Gaussian multiplicative chaos} 
The first result concerns the study of positive moments of multiplicative chaos, which plays an indispensable role in proving upper bounds for $\tilde{\eta}_q$ in all three regimes.

\begin{lem}\label{lem:1}
Let $u \in (-1,1)^d$. Let $s > 0, 0 < t \le 1$ and define
\begin{align*}
l : = s - d + \frac{\beta^2}{2}(t-1).
\end{align*}

\begin{itemize}
\item[(i)] If $l < 0$, then
\begin{align}\label{eq:lem1_1}
\EE \left[ \left( \int_{[-1,1]^d} \frac{e^{\beta X_\epsilon(x) - \frac{\beta^2}{2}\Var(X_\epsilon(x))}}{(|x-u|+\epsilon)^s}dx\right)^t \right] \le C.
\end{align}

\item[(ii)] If $l\ge0$, then
\begin{align} \label{eq:lem1_2}
\EE \left[ \left( \int_{[-1,1]^d} \frac{e^{\beta X_\epsilon(x) - \frac{\beta^2}{2}\Var(X_\epsilon(x))}}{(|x-u|+\epsilon)^s}dx\right)^t \right] \le C \epsilon^{-lt} \log (1/\epsilon).
\end{align}

\end{itemize}

All the constants $C>0$ above may depend on $d, \beta^2, s, t$ but not on $\epsilon \in (0, 1]$.

\end{lem}

\begin{proof}
This is a simple generalisation of \cite[Lemma C.1]{RV2014}. For simplicity we just treat the case where $u = 0$. If $\epsilon \ge \frac{1}{2}$, an easy application of Jensen's inequality suggests that
\begin{align*}
U_\epsilon 
&: = \EE \left[ \left( \int_{[-1,1]^d} \frac{e^{\beta X_\epsilon(x) - \frac{\beta^2}{2}\Var(X_\epsilon(x))}}{(|x|+\epsilon)^s}dx\right)^t \right] \le 2^{(s+d)t}.
\end{align*}

\noindent For $\epsilon < \frac{1}{2}$, pick $n \in \NN$ such that $2^n \epsilon \in [\frac{1}{2}, 1)$. Write $I_k = [-2^{-k}, 2^{-k}]^d$, we have
\begin{align*}
U_\epsilon 
& \le \EE \left[ \left( \int_{I_n} \frac{e^{\beta X_\epsilon(x) - \frac{\beta^2}{2}\Var(X_\epsilon(x))}}{(|x|+\epsilon)^s}dx\right)^t \right]
+ \sum_{k=1}^{n} \EE \left[ \left( \int_{I_{k-1} \setminus I_k} \frac{e^{\beta X_\epsilon(x) - \frac{\beta^2}{2}\Var(X_\epsilon(x))}}{(|x|+\epsilon)^s}dx\right)^t \right]\\
& \le 2^{(s-d)tn + st}\EE \left[ \left( \int_{I_0} e^{\beta X_\epsilon(x2^{-n}) - \frac{\beta^2}{2}\Var(X_\epsilon(x2^{-n}))}dx\right)^t \right]\\
& \qquad \qquad + 2^{dt}\sum_{k=1}^{n} 2^{(s-d)tk}\EE \left[ \left( \int_{I_{0}} e^{\beta X_\epsilon\left(x2^{-(k-1)}\right) - \frac{\beta^2}{2}\Var\left(X_\epsilon\left(x2^{-(k-1)}\right)\right)}dx\right)^t \right]\\
\end{align*}

\noindent Let $\Omega_k \sim N(0, k \log 2)$ and $Y \sim N(0, 2C_g)$ be independent of everything. Since
\begin{align*}
\EE \left[\left(X_\epsilon\left(x2^{-k}\right) + Y \right)\left(X_\epsilon\left( y2^{-k}\right) + Y \right)\right]
& = - \log \left(|x-y| + 2^k\epsilon\right) + k\log 2+ g_\epsilon(x, y) + 2C_g\\
& \ge \EE \left[\left(X_{2^k\epsilon}(x) + \Omega_k\right)\left(X_{2^k\epsilon}(y)+ \Omega_k\right)\right],
\end{align*}

\noindent it follows by Kahane's convexity inequality (Lemma \ref{lem:kahane}) that
\begin{align*}
& \EE \left[ \left( \int_{I_0} e^{\beta X_\epsilon(x2^{-k}) - \frac{\beta^2}{2}\Var(X_\epsilon(x2^{-k}))}dx\right)^t \right]\\
& \qquad \qquad = e^{-\beta^2(t-1)t C_g}\EE \left[ \left( \int_{I_0} e^{\beta \left(X_\epsilon(x2^{-k})+Y\right) - \frac{\beta^2}{2}\Var\left(X_\epsilon(x2^{-k})+Y\right)}dx\right)^t \right]\\
& \qquad \qquad \le e^{-\beta^2(t-1)t C_g} \EE \left[ \left( \int_{I_0} e^{\beta \left(X_{2^k\epsilon}(x) + \Omega_k\right) - \frac{\beta^2}{2}\Var\left(X_{2^k\epsilon}(x) + \Omega_k\right)}dx\right)^t \right]\\
& \qquad \qquad \le e^{-\beta^2(t-1)t C_g} 2^{dt} 2^{\frac{\beta^2}{2}(t-1) kt}
\end{align*}

\noindent and hence
\begin{align} \label{eq:bound}
U_\epsilon
& \le e^{-\beta^2(t-1)t C_g} 2^{(s+d)t}\left[ 2^{ltn} + \sum_{k=1}^n 2^{lt(k-1)}\right]
\le e^{-\beta^2(t-1)t C_g} 2^{(s+d)t}\left[ 2^{ltn} +\frac{2^{ltn}-1}{2^{lt} - 1} \right].
\end{align}

\noindent Therefore $U_\epsilon$ is uniformly bounded in $\epsilon \in (0,1]$ as soon as we have $2^{lt} < 1$, or $l<0$ which is the condition in the first statement. If $l \ge 0$, inequality \eqref{eq:bound} suggests that $U_\epsilon \lesssim n 2^{ltn}$. But then by construction we have $2^{tln} \le \epsilon^{-lt}$ and $n \le -\log_2 \epsilon$. This concludes the proof of the lemma.
\end{proof}

The second result is essentially the existence of negative moments of Gaussian multiplicative chaos, which is used in lower-bounding $\tilde{\eta}_q$ in the subcritical regimes $\beta^2 < 2d$.

\begin{lem}\label{lem:negm}
Let $\beta^2 < 2d$. Then for $q > 0$ there exists some constant $C > 0$ independent of $\epsilon \in (0, 1]$ such that
\begin{align*}
\EE \left[ \left(\int_{[-1,1]^d} e^{\beta X_\epsilon(x) - \frac{\beta^2}{2} \Var(X_\epsilon(x))}dx\right)^{-q}\right] \le C.
\end{align*}
\end{lem}

\begin{proof}
The result follows from a standard Gaussian comparison argument and we shall be brief here. Let $\tilde{X}$ be a centred log-correlated Gaussian field defined on $D \subset \RR^d$ with
\begin{align*}
\EE[\tilde{X}(x) \tilde{X}(y)] = -\log |x-y| + \tilde{g}(|x-y|), \qquad x, y \in D
\end{align*}

\noindent where $\tilde{g}$ is some continuous function on $\RR$. We remark that for any dimension $d$, such a log-correlated Gaussian field exists (see for instance star scale invariant covariance kernels in \cite[Section 2]{RV2014}). Let $\varphi$ be some non-negative continuous function supported on the unit ball centred at the origin such that $\varphi$ is positive definite and $\int_{\RR^d} \varphi(x)dx = 1$. Writing $\varphi_\epsilon = \epsilon^{-d}\varphi(\cdot/\epsilon)$, one can construct, for each $\epsilon \in (0, 1]$, an approximate field $\tilde{X}_\epsilon := \tilde{X} \ast \varphi_\epsilon$ with mean zero and covariance of the form
\begin{align*}
\EE[\tilde{X}_\epsilon(x) \tilde{X}_\epsilon(y)] = -\log (|x-y|+\epsilon) + \tilde{g}_\epsilon(x, y), \qquad x, y \in D.
\end{align*}

\noindent where $|\tilde{g}_\epsilon(x, y)| \le C_{\tilde{g}}$ for some $C_{\tilde{g}}>0$ independent of $x, y \in D$ and $\epsilon \in (0, 1]$.

When $\beta^2 < 2d$, the sequence of measures $\tilde{M}_{\beta, \epsilon}(dx) := e^{\beta \tilde{X}_\epsilon(x) - \frac{\beta^2}{2}\Var(\tilde{X}_\epsilon(x))}dx$ converges in distribution, as $\epsilon \to 0$, to some non-trivial limit $\tilde{M}_{\beta}(dx)$ in the space of Radon measures on $D$ equipped with the weak$^*$ topology (see e.g. \cite[Theorem 2.1]{RV2010} or \cite[Theorem 1.1]{B2017}). Moreover by e.g. \cite[Proposition 3.5]{RV2010},  we have, for any $q > 0$,
\begin{align*}
\EE\left[\tilde{M}_{\beta, \epsilon}([-1,1]^d)^{-q}\right] \xrightarrow{\epsilon \to 0} \EE\left[\tilde{M}_{\beta}([-1,1]^d)^{-q}\right] < \infty.
\end{align*}

\noindent Introducing a random variable $Y \sim N(0, C_g + C_{\tilde{g}})$ that is independent of $\tilde{X}$, we have
\begin{align*}
\EE[X_\epsilon(x) X_\epsilon(y)] \le \EE\left[ (\tilde{X}_\epsilon(x) + Y) (\tilde{X}_\epsilon(y) + Y)\right], \qquad \forall x, y \in D.
\end{align*}

\noindent Therefore, we conclude by Kahane's convexity inequality (Lemma \ref{lem:kahane}) that
\begin{align*}
\EE \left[ \left(\int_{[-1,1]^d} e^{\beta X_\epsilon(x) - \frac{\beta^2}{2} \Var(X_\epsilon(x))}dx\right)^{-q}\right]
& \le \EE \left[ \left(\int_{[-1,1]^d} e^{\beta (\tilde{X}_\epsilon(x)+Y) - \frac{\beta^2}{2} \Var(\tilde{X}_\epsilon(x) + Y)}dx\right)^{-q}\right]\\
& = e^{\frac{\beta^2 q}{2}(q+1)(C_g + C_{\tilde{g}})} \EE\left[\tilde{M}_{\beta, \epsilon}([-1,1]^d)^{-q}\right]
\end{align*}

\noindent is bounded uniformly in $\epsilon \in (0, 1]$.
\end{proof}

\subsection{Supremum of Gaussian processes}
The third result concerns the integrability of exponential moments of the supremum of Gaussian processes. As discussed earlier it is often useful to approximate some integrals by discrete sums. In order to pursue this idea in various places in the proof of the main theorem, we need the following discretisation control.

\begin{lem} \label{lem:2}
Let $\epsilon \in (0, 1]$. Partition $[-1,1]^d$ into boxes $\{B_i\}_i$ of equal width $\lceil \epsilon^{-1} \rceil^{-1}$, and let $r_i$ be the centre of box $B_i$. Then for any $c>0$ and $\kappa > 0$ there exists constant $C_{c, \kappa}$ independent of $\epsilon$ such that
\begin{align}\label{eq:lem2}
\EE\left[ \exp \left( c \sup_{i} \sup_{v \in B_i} \left|X_\epsilon(v) - X_\epsilon(r_i)\right|\right) \right] \le C_{c, \kappa} \epsilon^{-\kappa}.
\end{align}

\end{lem}

\begin{proof}
We reproduce the argument in the proof of \cite[Lemma 4.3]{RV2010} here for self-containedness. Since
\begin{align*}
& \EE\left[ \exp \left( c \sup_{i} \sup_{v \in B_i} \left|X_\epsilon(v) - X_\epsilon(r_i)\right|\right) \right]\\
& \qquad \le \EE\left[ \exp \left( c \sup_{i} \sup_{v \in B_i} \left(X_\epsilon(v) - X_\epsilon(r_i)\right)\right) \right]
+ \EE\left[ \exp \left( c \sup_{i} \sup_{v \in B_i} \left(- X_\epsilon(v) + X_\epsilon(r_i)\right)\right) \right],
\end{align*}

\noindent it suffices to consider $\sup_i \sup_{v \in B_i} (X_\epsilon(v) - X_\epsilon(r_i))$ without taking the absolute value as in \eqref{eq:lem2}. For simplicity assume that $\epsilon^{-1} \in \NN$. Let us write
\begin{align*}
W^i_\epsilon(v) = X_\epsilon(v) - X_\epsilon(r_i), \qquad v \in B_i.
\end{align*}

\noindent Note that for $v, v' \in B_i$, we have
\begin{align*}
\left|\EE[W^i_\epsilon(v) W^i_\epsilon(v')] \right|
&\le \frac{1}{2} \left(\EE[W_\epsilon^i(v)^2] + \EE[W_\epsilon^i(v')^2] \right)\\
& \le \log \left( 1 + \frac{|v - r_i|}{\epsilon}\right) + \log \left( 1 + \frac{|v' - r_i|}{\epsilon}\right) + 3C_g \le K
\end{align*}

\noindent for some constant $K>0$ independent of $\epsilon$. We define a centred Gaussian random variable $Y \sim N(0, K)$ that is independent of everything, and also new independent Gaussian processes $(\overline{W}_\epsilon^i(v))_{v \in B_i}$ having the same law as $(W_\epsilon^i(v) + Y)_{v \in B_i}$ for each $i$. Then by Slepian's Lemma (Corollary \ref{cor:Slepian}) we have

\begin{align*}
\EE \left[ \exp \left(c \sup_{i, v} W_\epsilon^i(v) \right) \right]
& \le e^{-\frac{c^2}{2} K} \EE \left[ \exp \left(c \sup_{i, v} \overline{W}_\epsilon^i(v) \right) \right].
\end{align*}

\noindent Let $\widetilde{W}_\epsilon^i : = \sup_{v \in B_i} \overline{W}_\epsilon^i(v) $. By Gaussian concentration (see e.g. \cite[Theorem 2.5.8]{GP}, we have, for each $i=1, \dots, \epsilon^{-d}$,
\begin{align*}
\PP \left( \left| \widetilde{W}_\epsilon^i - \EE\left[ \widetilde{W}_\epsilon^i\right] \right| > u \right) \le 2e^{-\frac{u^2}{2\sigma_i^2}}
\end{align*}

\noindent where $\sigma_i^2 := \sup_{v \in B_i} \EE\left[\overline{W}_\epsilon^i(v)^2\right]$. It is not hard to show that both $\sigma_i^2$ and $\EE\left[ \widetilde{W}_\epsilon^i \right]$ can be upper bounded uniformly in $\epsilon$ and $i$, and therefore we have for some $C > 0$
\begin{align*}
\PP\left( \widetilde{W}_\epsilon^i > u \right) \le C e^{-\frac{u^2}{C}}.
\end{align*}

\noindent Applying union bound, 
\begin{align*}
\PP\left( \max_i \widetilde{W}_\epsilon^i > u + \sqrt{-C\log \epsilon^d}\right) 
& \le \epsilon^{-d} \PP\left(\widetilde{W}_\epsilon^i > u + \sqrt{-C\log \epsilon^d}\right)
\le Ce^{-\frac{u^2}{C}}
\end{align*}

\noindent and hence we obtain
\begin{align*}
\EE \left[ \exp \left(c \sup_i \widetilde{W}_\epsilon^i \right) \right] \le C' e^{C' \sqrt{-\log \epsilon^d}} = C' \epsilon^{-\frac{C' \sqrt{d}}{\sqrt{-\log \epsilon}}} = o(\epsilon^{-\kappa})
\end{align*}

\noindent for any $\kappa > 0$, as claimed. 

\end{proof}

%----------------------------------------------------------------------------------------
%	First regime
%----------------------------------------------------------------------------------------

\section{Proof of Theorem \ref{theo:main}} \label{sec:proof}
\subsection{The high temperature regime}
In the regime where $0 \le \beta^2 < \frac{2d}{2q-1}$ we would like to show that
\begin{align}\label{eq:pred1}
\tilde{\eta}_q = -\frac{\beta^2q^2}{2} + \frac{\beta^2 q}{2}.
\end{align}

\subsubsection{Lower bound for $\tilde{\eta}_q$ in the high temperature regime}

%----------------------------------------------------------------------------------------
%	First regime: lower bound construction
%----------------------------------------------------------------------------------------
Finding a lower bound for $\tilde{\eta}_q(\beta)$ is equivalent to finding an upper bound for $\EE\left[ \frac{Z_\epsilon(\beta q)}{Z_\epsilon(\beta)^q} \right]$. By Lemma \ref{lem:Girsanov} , it suffices to show that
\begin{align*}
\EE \left[\left(\int_{[-1,1]^d} \frac{e^{\beta X_\epsilon(x) - \frac{\beta^2}{2} \Var\left(X_\epsilon(x)\right)}}{(|x-u| + c\epsilon)^{\beta^2 q}}dx \right)^{-q}\right]du
\end{align*}

\noindent is bounded above uniformly in $\epsilon \in [0, 1]$ and in $u \in [-1,1]^d$. But then
\begin{align*}
& \EE \left[\left(\int_{[-1,1]^d} \frac{e^{\beta X_\epsilon(x) - \frac{\beta^2}{2} \Var\left(X_\epsilon(x)\right)}}{(|x-u| + \epsilon)^{\beta^2 q}}dx \right)^{-q}\right]du\\
& \qquad \le \left(\min_{x, u \in [-1,1]^d} (|x-u|+\epsilon)^{-\beta^2 q}\right)^{-q}
\EE \left[\left(\int_{[-1,1]^d} e^{\beta X_\epsilon(x) - \frac{\beta^2}{2} \Var\left(X_\epsilon(x)\right)}dx \right)^{-q}\right]du\\
& \qquad \le C
\end{align*}

\noindent where the last inequality follows from Lemma \ref{lem:negm}. Putting everything together, we have
\begin{align*}
\EE\left[ \frac{Z_\epsilon(\beta q)}{Z_\epsilon(\beta)^q} \right] 
&\lesssim \epsilon^{-\frac{\beta^2 q^2}{2} + \frac{\beta^2 q}{2}}, 
\qquad \liminf_{\epsilon \to 0} \frac{\log \EE\left[ \frac{Z_\epsilon(\beta q)}{Z_\epsilon(\beta)^q} \right]}{\log \epsilon} \ge -\frac{\beta^2 q^2}{2} + \frac{\beta^2 q}{2}
\end{align*}

\noindent which is a lower bound that matches \eqref{eq:pred1}.

\subsubsection{Upper bound for $\tilde{\eta}_q$ in high temperature regime}
Picking $m > 0$, we have by Jensen's inequality that
\begin{align*}
& \EE \left[ \left( \int_{[-1,1]^d} \frac{e^{\beta X_\epsilon(x)-\frac{\beta^2}{2} \Var(X_\epsilon(x))}}{(|x-u|+\epsilon)^{\beta^2 q}}dx\right)^{-q}\right]
\ge \EE \left[ \left( \int_{[-1,1]^d} \frac{e^{\beta X_\epsilon(x)-\frac{\beta^2}{2} \Var(X_\epsilon(x))}}{(|x-u|+\epsilon)^{\beta^2 q}}dx\right)^{q/m}\right]^{-m}\\
\end{align*}

\noindent which is lower-bounded by a constant independent of $\epsilon$ by Lemma \ref{lem:1}(i) when
\begin{align}\label{eq:beta1p}
\beta^2 < \frac{d}{q\left(1 + \frac{1}{2m}\right) - \frac{1}{2}}.
\end{align}

\noindent Combining this with Lemma \ref{lem:Girsanov}, we obtain
\begin{align}
\EE \left[\frac{Z_\epsilon(\beta q)}{Z_\epsilon(\beta)^q} \right]
\label{eq:eta1u}
& \gtrsim  \epsilon^{-\frac{\beta^2 q^2}{2} + \frac{\beta^2 q}{2}},
\qquad  \limsup_{\epsilon \to 0} \frac{\log \EE\left[ \frac{Z_\epsilon(\beta q)}{Z_\epsilon(\beta)^q} \right]}{\log \epsilon} \le-\frac{\beta^2q^2}{2} +\frac{\beta^2 q}{2}
\end{align}

\noindent so long as \eqref{eq:beta1p} holds. Since $m > 0$ is arbitrary, we conclude that \eqref{eq:eta1u} holds for any
\begin{align*} %\label{eq:beta1}
\beta^2 < \frac{d}{q - \frac{1}{2}} = \frac{2d}{2q-1}.
\end{align*}

\subsection{Intermediate regime}
In the regime where $\frac{2d}{2q-1} \le \beta^2 < 2d,$ we would like to prove that
\begin{align}\label{eq:pred2}
\tilde{\eta}_q = \frac{(2d - \beta^2)^2}{8\beta^2} - d(q-1).
\end{align}

%----------------------------------------------------------------------------------------
%	Intermediate regime: lower bound construction
%----------------------------------------------------------------------------------------

\subsubsection{Lower bound for $\tilde{\eta}_q$ in the intermediate regime}

Recall the magical choice $c = \frac{1}{q} \left(\frac{1}{2} + \frac{d}{\beta^2}\right)$ in Section \ref{sec:2}. For the purpose of proving lower bound for $\tilde{\eta}_q$, we need to show that the remaining integral in \eqref{eq:Girsanov} is irrelevant. For this, we claim that

\begin{lem} \label{lem:r2l}
If $\frac{2d}{2q-1} \le \beta^2 < 2d$ and $c = \frac{1}{q} \left(\frac{1}{2} + \frac{d}{\beta^2}\right)$,  then 
\begin{align*}
\EE \left[\left(\frac{e^{\beta (1-c) X_{\epsilon}(u)}}{\int_{[-1,1]^d} \frac{e^{\beta X_\epsilon(x) - \frac{\beta^2}{2} \Var \left(X_\epsilon(x)\right)}}{(|x-u| + \epsilon)^{\beta^2 q c}}dx}\right)^q\right] = O(\epsilon^{-\kappa})
\end{align*}

\noindent for any $\kappa > 0$.
\end{lem}
\begin{proof}
Consider the event $A = \{ X_\epsilon(u) > 0\}$. Based on this event we split our expectation into two parts:
\begin{align*}
& \EE \left[\left(\frac{e^{\beta (1-c) X_{\epsilon}(u)}}{\int_{[-1,1]^d} \frac{e^{\beta X_\epsilon(x) - \frac{\beta^2}{2} \Var \left(X_\epsilon(x)\right)}}{(|x-u| + \epsilon)^{\beta^2 q c}}dx}\right)^q\right]\\
& = \EE \left[\left(\frac{e^{\beta (1-c) X_{\epsilon}(u)}}{\int_{[-1,1]^d} \frac{e^{\beta X_\epsilon(x) - \frac{\beta^2}{2} \Var \left(X_\epsilon(x)\right)}}{(|x-u| + \epsilon)^{\beta^2 q c}}dx}\right)^q 1_A\right]
+ \EE \left[\left(\frac{e^{\beta (1-c) X_{\epsilon}(u)}}{\int_{[-1,1]^d} \frac{e^{\beta X_\epsilon(x) - \frac{\beta^2}{2} \Var \left(X_\epsilon(x)\right)}}{(|x-u| + \epsilon)^{\beta^2 q c}}dx}\right)^q 1_{A^c}\right].
\end{align*}

\noindent It should not be surprising that the second term after the equality above is negligible. Indeed since $c = \frac{1}{q} \left(\frac{1}{2} + \frac{d}{\beta^2}\right) \in (0, 1]$, we have
\begin{align*}
\EE \left[\left(\frac{e^{\beta (1-c) X_{\epsilon}(u)}}{\int_{[-1,1]^d} \frac{e^{\beta X_\epsilon(x) - \frac{\beta^2}{2} \Var \left(X_\epsilon(x)\right)}}{(|x-u| + \epsilon)^{\beta^2 q c}}dx}\right)^q 1_{A^c}\right]
& \le  3^{\beta^2 q^2 c}  \EE \left[\left(\int_{[-1,1]^d} e^{\beta X_\epsilon(x) - \frac{\beta^2}{2} \Var \left(X_\epsilon(x)\right)}dx\right)^{-q}\right]\\
& = O(1)
\end{align*}

\noindent again due to the existence of the negative moments of subcritical Gaussian multiplicative chaos (Lemma \ref{lem:negm}).  We are then left with upper-bounding the remaining term.\\

To proceed, we now adopt the notations in the proof of Lemma \ref{lem:2}, partitioning $[-1,1]^d$ into boxes $\{B_i\}_i$ of equal width $\epsilon^{-1}$ and denoting by $r_i$ the centre of box $B_i$. Say $u \in B_j$ for some $j$. Note that $x$ is a point in the box $B_i$, we have the simple inequality
\begin{align*}
|x-u| + \epsilon
&\le |x - r_i| + |u - r_j| + |r_i - r_j| + \epsilon\\
& \le (1+\sqrt{d}) (|r_i - r_j| + \epsilon)
\end{align*}

\noindent because $|r_i - r_j| \ge \epsilon$ while $\max(|x-r_i|, |u-r_j|) \le \epsilon\sqrt{d}$. Therefore
\begin{align*}
& \EE \left[\left(\frac{e^{\beta (1-c) X_{\epsilon}(u)}}{\int_{[-1,1]^d} \frac{e^{\beta X_\epsilon(x) - \frac{\beta^2}{2} \Var \left(X_\epsilon(x)\right)}}{(|x-u| + \epsilon)^{\beta^2 q c}}dx}\right)^q 1_A\right]\\
& \lesssim  \epsilon^{- \frac{\beta^2 q}{2}}\EE \left[\left(\frac{e^{\beta(1-c) (X_\epsilon(u) - X_\epsilon(r_j))} e^{\beta(1-c) X_{\epsilon}(r_j)}}{e^{\beta \inf_i \inf_{x \in B_i} (X_\epsilon(x) - X_\epsilon(r_i))}\sum_i \epsilon^d \frac{e^{\beta X_\epsilon(r_i)}}{(|r_i-r_j| + \epsilon)^{\beta^2 q c}}}\right)^q  1_A\right]\\
& \lesssim  \epsilon^{- \frac{\beta^2 q}{2}+\beta^2 q^2 c(1-c)}\EE \left[\left(\frac{e^{\beta(1-c) (X_\epsilon(u) - X_\epsilon(r_j))}\left(e^{\beta X_{\epsilon}(r_j)}\epsilon^{-\beta^2 q c}\right)^{1-c}}{e^{\beta \inf_i \inf_{x \in B_i} (X_\epsilon(x) - X_\epsilon(r_i))}\sum_i \epsilon^d \frac{e^{\beta X_\epsilon(r_i)}}{(|r_i-r_j| + \epsilon)^{\beta^2 q c}}}\right)^q  1_A\right]\\
& \lesssim  \epsilon^{- \frac{\beta^2 q}{2}+ \beta^2 q^2 c(1-c) -dq}\\
& \qquad \qquad \times \EE \left[\frac{e^{\beta q (1-c) (X_\epsilon(u) - X_\epsilon(r_j))}}{e^{\beta q \inf_i \inf_{x \in B_i} (X_\epsilon(x) - X_\epsilon(r_i))}} \left(e^{\beta X_{\epsilon}(r_j)}\epsilon^{-\beta^2 q c}\right)^{-qc} \left(\frac{e^{\beta X_{\epsilon}(r_j)}\epsilon^{-\beta^2 q c}}{\sum_i \frac{e^{\beta X_\epsilon(r_i)}}{(|r_i-r_j| + \epsilon)^{\beta^2 q c}}}\right)^q1_A\right]\\
& \le \epsilon^{\beta^2 q^2 \left[c - \frac{1}{q} \left( \frac{1}{2} + \frac{d}{\beta^2}\right)\right]} \EE \left[ \frac{e^{-\beta qc X_\epsilon(u)} e^{\beta q (X_\epsilon(u) - X_\epsilon(r_j))}}{e^{\beta q \inf_i \inf_{x \in B_i} (X_\epsilon(x) - X_\epsilon(r_i))}} 1_A\right]\\
& \le \EE \left[e^{2\beta q \sup_i \sup_{x \in B_i} \left|X_\epsilon(x) - X_\epsilon(r_i)\right|}\right]\\
& \lesssim \epsilon^{-\kappa}
\end{align*}

\noindent where $\kappa > 0$ is arbitrary from Lemma \ref{lem:2}, as claimed.
\end{proof}

Going back to the proof of lower bound of $\tilde{\eta}_q$, by choosing $c = \frac{1}{q} \left(\frac{1}{2} + \frac{d}{\beta^2}\right)$ and applying Lemma \ref{lem:r2l} we are able to obtain
\begin{align*}
\EE\left[ \frac{Z_\epsilon(\beta q)}{Z_\epsilon(\beta)^q} \right] 
&\lesssim \epsilon^{\frac{(2d-\beta^2)^2}{8\beta^2} - d(q-1) - \kappa}, 
\qquad \liminf_{\epsilon \to 0} \frac{\log \EE\left[ \frac{Z_\epsilon(\beta q)}{Z_\epsilon(\beta)^q} \right]}{\log \epsilon} \ge \frac{(2d-\beta^2)^2}{8\beta^2} - d(q-1) - \kappa.
\end{align*}

\noindent Since $\kappa > 0$ is arbitrary, this is a lower bound that matches \eqref{eq:pred2}.

%----------------------------------------------------------------------------------------
%	Intermediate regime: upper bound construction
%----------------------------------------------------------------------------------------

\subsubsection{Upper bound for $\tilde{\eta}_q$ in the intermediate regime}
For the upper bound we shall actually defer our choice of the parameter $c$, and start by finding a lower bound for
\begin{align*}
\EE \left[\frac{e^{\beta q(1-c) X_{\epsilon}(u)}}{\left(\int_{[-1,1]^d} \frac{e^{\beta X_\epsilon(x) - \frac{\beta^2}{2} \Var\left(X_\epsilon(x)\right)}}{(|x-u| + \epsilon)^{\beta^2 q c}}dx \right)^q}\right].
\end{align*}

\noindent We apply the reverse H\"older's inequality
\begin{align*}
||fg||_1 \ge ||f||_\gamma ||g||_\delta,
\quad \gamma < 0, \quad \delta \in (0,1), \quad \frac{1}{\gamma} + \frac{1}{\delta} = 1,
\end{align*}

\noindent and obtain
\begin{align*}
& \EE \left[\frac{e^{\beta q(1-c) X_{\epsilon}(u)}}{\left(\int_{[-1,1]^d} \frac{e^{\beta X_\epsilon(x) - \frac{\beta^2}{2} \Var\left(X_\epsilon(x)\right)}}{(|x-u| + \epsilon)^{\beta^2 q c}}dx \right)^q}\right]\\
& \qquad  \gtrsim \EE \left[e^{\beta q(1-c)\gamma X_{\epsilon}(u)}\right]^{\frac{1}{\gamma}}
\EE \left[\left(\int_{[-1,1]^d} \frac{e^{\beta X_\epsilon(x)- \frac{\beta^2}{2} \Var\left(X_\epsilon(x)\right)}}{(|x-u| + \epsilon)^{\beta^2 q c}}dx \right)^{-q\delta}\right]^{\frac{1}{\delta}}\\
& \qquad \gtrsim \epsilon^{- \frac{\beta^2 q^2 (1-c)^2 \gamma}{2}}
\EE \left[\left(\int_{[-1,1]^d} \frac{e^{\beta X_\epsilon(x) - \frac{\beta^2}{2} \Var\left(X_\epsilon(x)\right)}}{(|x-u| + \epsilon)^{\beta^2 q c}}dx \right)^{-q\delta}\right]^{\frac{1}{\delta}}\\
& \qquad \ge  \epsilon^{- \frac{\beta^2 q^2 (1-c)^2 \gamma}{2}}
\EE \left[\left(\int_{[-1,1]^d} \frac{e^{\beta X_\epsilon(x) - \frac{\beta^2}{2}\Var(X_\epsilon(x))}}{(|x-u| + \epsilon)^{\beta^2 q c}}dx \right)^{\frac{q\delta}{m}}\right]^{-\frac{m}{\delta}}
\end{align*}

\noindent where in the last line we introduced $m>0$ and applied Jensen's inequality. We shall assume that $m$ is such that $\frac{q\delta}{m} \in (0, 1)$. Combining everything we have

\begin{align} \label{eq:ub}
\EE\left[ \frac{Z_\epsilon(\beta q)}{Z_\epsilon(\beta)^q} \right]
\gtrsim \epsilon^{-\frac{\beta^2 q^2}{2} [1 - (1-c)^2 (1-\gamma)] + \frac{\beta^2 q}{2}} \EE \left[\left(\int_{[-1,1]^d} \frac{e^{\beta X_\epsilon(x) - \frac{\beta^2}{2}\Var(X_\epsilon(x))}}{(|x-u| + \epsilon)^{\beta^2 q c}}dx \right)^{\frac{q\delta}{m}}\right]^{-\frac{m}{\delta}}
\end{align}

Let us for a moment assume that the expectation
\begin{align} \label{eq:r2ire}
\EE \left[\left(\int_{[-1,1]^d} \frac{e^{\beta X_\epsilon(x) - \frac{\beta^2}{2}\Var(X_\epsilon(x))}}{(|x-u| + \epsilon)^{\beta^2 q c}}dx \right)^{\frac{q\delta}{m}}\right]
\end{align}

\noindent is irrelevant. In order to get the desired upper bound for $\tilde{\eta}_q$ we perform the following matching: introduce some $\kappa_0 > 0$ and we ask if it is possible to choose some $c \in (0, 1)$ such that
\begin{align*}
-\frac{\beta^2 q^2}{2} \left[1-(1-c)^2 (1-\gamma)\right] + \frac{\beta^2q}{2} &= \frac{(2d-\beta^2)^2}{8\beta^2} - d(q-1) + \kappa_0.
\end{align*}

\noindent This is equivalent to
\begin{align}
\notag (1-c)^2 (1-\gamma) &= \frac{(1-c)^2}{1-\delta}
= \frac{1}{q^2} \left[\left(q - \frac{1}{2}\right) - \frac{d}{\beta^2}\right]^2 + \kappa_1, \\
\label{eq:eta2_choice} 
\Rightarrow c &= 1 - \sqrt{1-\delta} \left[\frac{1}{q} \left(q - \frac{1}{2} - \frac{d}{\beta^2} + \kappa_2\right)  \right] \in (0, 1)
\end{align}

\noindent where $\kappa_1, \kappa_2 > 0$ are some constants such that $\kappa_1, \kappa_2 \to 0^+$ as $\kappa_0 \to 0^+$.

In order for this matching to work, we need to go back and check that the remaining expectation \eqref{eq:r2ire} is upper bounded for such a choice of $c$, so that when it appears in \eqref{eq:ub} it is not making additional contribution. By Lemma \ref{lem:1}(i) this requires

\begin{align*}
\beta^2 qc - d + \frac{1}{2} \left( \frac{q \delta}{m} - 1\right) \beta^2 < 0
\qquad  \Rightarrow \qquad \beta^2 \left[q\left(c + \frac{\delta}{2m}\right) - \frac{1}{2}\right] < d.
\end{align*}

\noindent Since $m > 0$ can be made arbitrarily large, it suffices to check that
\begin{align}\label{eq:beta2}
\beta^2 \left(qc - \frac{1}{2}\right) < d.
\end{align}

\noindent But
\begin{align*}
\beta^2 \left(qc - \frac{1}{2}\right)
& = \beta^2 \left[ q - \frac{1}{2} - \sqrt{1-\delta} \left(q - \frac{1}{2} - \frac{d}{\beta^2} + \kappa_2 \right)\right]\\
& = d + \underbrace{\left(1-\sqrt{1-\delta}\right) \left[ \beta^2 \left(q - \frac{1}{2}\right) - d \right] - \beta^2 \sqrt{1-\delta} \kappa_2}_{=:(*)}. 
\end{align*}

\noindent Regardless of the value of $q, \beta^2$ and $\kappa_0$ (and thus $\kappa_2$), one can always choose $\delta > 0$ close enough to zero such that $(*)$ is negative. Hence \eqref{eq:beta2} is satisfied and we obtain
\begin{align*}
\limsup_{\epsilon \to 0} \frac{\log \EE\left[ \frac{Z_\epsilon(\beta q)}{Z_\epsilon(\beta)^q} \right]}{\log \epsilon}\le \frac{(2d-\beta^2)^2}{8\beta^2} - (q-1) + \kappa_0.
\end{align*}

\noindent Since $\kappa_0 > 0$ is arbitrary, we have proved the matching upper bound for $\tilde{\eta}_q$ in the intermediate regime.

\subsection{Low temperature regime}
In the so-called freezing regime where $\beta^2 > 2d,$ we shall show that
\begin{align} \label{eq:pred3}
\tilde{\eta}_q = -d(q-1).
\end{align}

\subsubsection{Lower bound for $\tilde{\eta}_q$ in the low temperature regime}
We exploit our knowledge that the Boltzmann-Gibbs measure is localised in the freezing regime. Unlike the intermediate regime where Cameron-Martin theorem had to be employed first, here we may perform discretisation directly:

\begin{align*}
\EE\left[ \frac{Z_\epsilon(\beta q)}{Z_\epsilon(\beta)^q}\right]
& = \EE\left[ \frac{\int_{[-1,1]^d} e^{\beta q X_\epsilon(u)}du}{\left(\int_{[-1,1]^d} e^{\beta X_\epsilon(x)}dx\right)^q}\right]\\
& \le \EE\left[ e^{2\beta q \sup_{k}\sup_{v \in B_k} |X_\epsilon(v) - X_\epsilon(r_k)|}\frac{\sum_{j} \epsilon^d e^{\beta q X_\epsilon(r_j)}}{\left(\sum_i \epsilon^{d}e^{\beta X_\epsilon(r_i)}\right)^q} \right]\\
&= \epsilon^{-d(q-1)}\EE\left[ e^{2\beta q \sup_{k}\sup_{v \in B_k} |X_\epsilon(v) - X_\epsilon(r_k)|}\sum_j \left(\frac{e^{\beta X_\epsilon(r_j)}}{\sum_i e^{\beta X_\epsilon(r_i)}}\right)^q \right]\\
&\le \epsilon^{-d(q-1)}\EE\left[e^{2\beta q \sup_{|u-v| \le \epsilon} (X_\epsilon(u) - X_\epsilon(v))}\right]\\
& \lesssim \epsilon^{-d(q-1) - \kappa}
\end{align*}

\noindent again by Lemma \ref{lem:2}. Therefore
\begin{align*}
\liminf_{\epsilon \to 0} \frac{\log \EE \left[ \frac{Z_\epsilon(q \beta)}{Z_\epsilon(\beta)^q} \right]}{\log \epsilon} \ge -d(q-1) - \kappa.
\end{align*}

\noindent Since $\kappa > 0$ is arbitrary, we obtain our desired lower bound for $\tilde{\eta}_q$.

\subsubsection{Upper bound for $\tilde{\eta}_q$ in the low temperature regime}
For upper bound we follow the similar approach in the intermediate regime and use the estimate.
\begin{align}
\EE\left[ \frac{Z_\epsilon(\beta q)}{Z_\epsilon(\beta)^q} \right] 
& \gtrsim \epsilon^{- \frac{\beta^2 q^2}{2}\left[1 - (1-c)^2 (1-\gamma)\right] +\frac{\beta^2q}{2}}
\int_{[-1,1]^d} \EE \left[\left(\int_{[-1,1]^d} \frac{e^{\beta X_\epsilon(x) - \frac{\beta^2}{2}\Var(X_\epsilon(x))}}{(|x-u| + \epsilon)^{\beta^2 q c}}dx \right)^{\frac{q\delta}{m}}\right]^{-\frac{m}{\delta}}du.
\end{align}

\noindent To find a suitable value of $c$ that gives us the right order, this time we need to apply Lemma \ref{lem:1}(ii) instead: assuming that $c, m, \delta$ are chosen such that
\begin{align}
\label{eq:ldef}
l:=  \beta^2 \left[q \left(c + \frac{\delta}{2m}\right)-\frac{1}{2}\right] - d > 0, 
\end{align}

\noindent we obtain, for any $\kappa > 0$,
\begin{align*}
\EE\left[ \frac{Z_\epsilon(\beta q)}{Z_\epsilon(\beta)^q} \right]
& \gtrsim \epsilon^{- \frac{\beta^2 q^2}{2}\left[1 - (1-c)^2 (1-\gamma)\right] +\frac{\beta^2q}{2} + ql + \kappa}, \\
\qquad \limsup_{\epsilon \to 0} \frac{\log \EE\left[ \frac{Z_\epsilon(\beta q)}{Z_\epsilon(\beta)^q} \right]}{\log \epsilon} & \le - \frac{\beta^2 q^2}{2}\left[1 - (1-c)^2 (1-\gamma)\right] +\frac{\beta^2q}{2} + ql + \kappa.
\end{align*}

\noindent Since $\kappa>0$ is arbitrary, we have
\begin{align*}
\limsup_{\epsilon \to 0} \frac{\log \EE\left[ \frac{Z_\epsilon(\beta q)}{Z_\epsilon(\beta)^q} \right]}{\log \epsilon} & \le - \frac{\beta^2 q^2}{2}\left[1 - (1-c)^2 (1-\gamma)\right] +\frac{\beta^2q}{2} + ql
\end{align*}

\noindent and we want to see whether it is possible to match the exponent:
\begin{align*}
- \frac{\beta^2 q^2}{2}\left[1 - (1-c)^2 (1-\gamma)\right] +\frac{\beta^2q}{2} + ql = -d(q-1),
\end{align*}

\noindent which is equivalent to
\begin{align*} 
\notag 0 
& = \frac{\beta^2 q^2}{2} \left[1 - (1-c)^2(1-\gamma) -2\left(c + \frac{\delta}{2m}\right)\right] + d\\
\notag
& = \frac{\beta^2 q^2}{2} \left[-(1-\gamma) \left(c + \frac{\gamma}{1-\gamma}\right)^2 + \frac{\gamma^2}{1-\gamma}+\gamma -\frac{\delta}{m}\right] + d\\
& = \frac{\beta^2 q^2}{2} \left[-\frac{(c-\delta)^2}{1-\delta} + \delta\left(1 -\frac{1}{m}\right)\right] + d
\end{align*}

\noindent Solving for $c$ and, say, taking the larger root, we have
\begin{align} \label{eq:ub3con}
c &= \delta + \sqrt{1-\delta} \sqrt{\frac{2d}{\beta^2 q^2} + \delta \left(1 - \frac{1}{m}\right)}, \qquad \delta \in (0, 1).
\end{align}

\noindent We still have to verify \eqref{eq:ldef} for such a choice of $c$, but then when $\delta = 1$ we have $c = 1$ and $l > 0$. Hence by continuity we can find $\delta$ smaller than but close enough to $1$ such that $l > 0$, and we conclude that
\begin{align*}
\limsup_{\epsilon \to 0} \frac{\log \EE\left[ \frac{Z_\epsilon(\beta q)}{Z_\epsilon(\beta)^q} \right]}{\log \epsilon} \le -d(q-1).
\end{align*}

\begin{appendices}
\section{Gaussian toolbox}

\begin{lem}[Kahane's convexity inequality]\label{lem:kahane}
Let  $(X_i)$ and $(Y_i)$ be two centred Gaussian vectors s.t.
\begin{align*}
\EE[X_i X_j] \le \EE[Y_i Y_j].
\end{align*}

\noindent Then for any choice of non-negative weights $(p_i)$ and  all convex  functions $F: \RR_+ \to \RR$ with at most polynomial growth at infinity, we have

\begin{align}
\EE\left[F\left(\sum_{i=1}^n p_i e^{X_i - \frac{1}{2}\EE[X_i^2]}\right)\right] \le \EE\left[F\left(\sum_{i=1}^n p_i e^{Y_i - \frac{1}{2}\EE[Y_i^2]}\right)\right].
\end{align}
\end{lem}

Slepian's lemma may then be obtained as a corollary of the above result (see e.g. \cite[Corollary 6.3]{RV2010}).
\begin{cor}[Slepian's lemma]\label{cor:Slepian}
Let $(X_i)$ and $(Y_i)$ be two centred Gaussian vectors such that
\begin{itemize}
\item $\EE[X_i^2] = \EE[Y_i^2]$ for every $i$, and
\item $\EE[X_i X_j] \le \EE[Y_i Y_j]$ for every $i \ne j$.
\end{itemize}

\noindent Then
\begin{align*}
\PP \left( \sup_{1 \le i \le n} X_i < x \right) \le \PP \left( \sup_{1 \le i \le n} Y_i < x \right).
\end{align*}

\noindent In particular, for any increasing $F: \RR \to \RR_+$ we obtain
\begin{align}
\EE \left[ F\left(\sup_{1 \le i \le n} Y_i \right)\right] \le \EE \left[ F\left(\sup_{1 \le i \le n} X_i \right)\right].
\end{align}

\end{cor}

\end{appendices}

\end{document}